\documentclass[12pt,reqno]{amsart}

\usepackage{amsmath,amsfonts,amssymb,enumerate}
\usepackage{amsthm,dsfont}
\usepackage{latexsym, mathtools,verbatim,graphicx}
\usepackage{amscd,hyperref}
\usepackage{bbold,bbm}
\usepackage{color}
\definecolor{burntorange}{rgb}{0.8, 0.33, 0.0}

\theoremstyle{plain}
\newtheorem{theorem}{Theorem}[section]
\newtheorem{thm}[theorem]{Theorem}

\newtheorem{corollary}[theorem]{Corollary}

\newtheorem{prop}[theorem]{Proposition}
\theoremstyle{definition}

\newtheorem{defn}[theorem]{Definition}

\theoremstyle{remark}
\newtheorem{rem}[theorem]{Remark}

\newcommand{\R}{\mathbb{R}}
\newcommand{\C}{\mathbb{C}}

\newcommand{\HH}{\mathcal{H}}
\newcommand{\T}{\mathbb{T}}
\newcommand{\D}{\mathbb{D}}
\newcommand{\N}{\mathbb{N}}

\DeclareMathOperator*{\Span}{span}
\newcommand{\cB}{\mathcal{B}}

\begin{document}
\title[Convergence and cyclicity]{Convergence and preservation of cyclicity}
\author[Aguilera]{Alejandra Aguilera}

\address{Alejandra Aguilera \newline Departamento de Matem\'atica, Universidad de Buenos Aires, \newline Instituto de Matem\'atica ``Luis Santal\'o'' (IMAS-CONICET-UBA), \newline Buenos Aires, Argentina} \email{aaguilera@dm.uba.ar}

\author[Seco]{Daniel Seco}
\address{Daniel Seco \newline  Universidad de la Laguna, Universidad Carlos III de Madrid and Instituto de Ciencias Matem\'aticas (CSIC-UAM-UC3M-UCM)  \newline  Avenida Astrof\'isico Francisco S\'anchez, s/n.  \newline Facultad de Ciencias, secci\'on: Matem\'aticas, apdo. 456.  \newline 38200 San Crist\'obal de La Laguna \newline
Santa Cruz de Tenerife,  Spain} \email{dsecofor@ull.edu.es}

\begin{abstract}
We prove that the set of cyclic (resp., non-cyclic) functions in Dirichlet type spaces $D_{\alpha}$ is not closed in the topology induced by the norm. We also show that some additional conditions on a convergent sequence of cyclic functions $\{f_n\}$ force cyclicity of the limit $f$. One of these conditions is the existence of a uniform bound on the multiplier norm of $f/f_n$. We show counterexamples satisfying all the same conditions except that $f/f_n$ are multipliers without uniform norm bounds. Then we find precise estimates for the distance between the corresponding optimal polynomial approximants of each degree $d$ ($\{p_{n,d}\}$ and $p_d$ associated to the sequence $\{f_n\}$ and to $f$), and show how some of the constants providing a control of $\|p_{n,d}-p_d\|$ blow up as the choice of $f$ moves within $D_\alpha$. 
\end{abstract}

\thanks{This project has received funding: from the European Union's Horizon 2020 research and innovation programme under the Marie Sk\l{}odowska-Curie grant agreement No. 777822; from the Spanish Ministry of Economy and Competitiveness, through the “Severo Ochoa Programme for Centers of Excellence in R$\&$D” (CEX2019-000904-S); through grant PID2019-106433GB-I00, from Agencia Estatal de Investigaci\'on, Spain, MCIN/AEI/10.13039/501100011033; and from the Madrid Government (Comunidad de Madrid-Spain) under the Multiannual Agreement with UC3M in the line of Excellence of University Professors
(EPUC3M23), and in the context of the V PRICIT (Regional Programme of Research and Technological Innovation).
}



\date{\today}

\maketitle

\section{Introduction}

Let $\mathcal{H}$ be a Hilbert space of analytic functions where polynomials are dense.
	\begin{defn}\label{def1}
		A function $f\in \mathcal{H}$ is \emph{cyclic} if 
		\[[f]_{\mathcal{H}} := \overline{ \Span\{ z^{k}f:\, k\in\N \} } = \mathcal{H}.\]	\end{defn}
	
	 Since the function $1$ is cyclic, one can easily find an equivalent condition for cyclicity of a function: $f\in \mathcal{H}$ is cyclic when there exists a sequence of polynomials $\{p_{n}\}_{n=0}^{\infty}$ such that 
	    \begin{equation}\label{eqn2001}\|p_{n}f - 1\|_{\alpha} \to 0 \quad \text{as }n\to\infty.\end{equation}

Cyclic functions play an important role in complex analysis since knowing that a function is cyclic may allow one to reduce the proof that some property $P$ holds for all elements of an infinite-dimensional space by showing, instead, that $P$ is preserved by convergence, it is linear and it holds for a cyclic function. Recently, some applications of cyclic functions in signal theory have been studied \cite{BenCen,SargentSola}.

Characterizing all cyclic functions in a space is usually a difficult problem. Thus, in several attempts to do so, different authors have started by describing cyclic \emph{polynomials} with the hope of extending later the conclusions to more general functions \cite{BS84,TAMS16}. In this sense, one could be tempted to think that cyclicity (or lack thereof) should be preserved by norm convergence. In this article, we will start by showing that this is \emph{not the case} in a relevant family of function spaces called \emph{Dirichlet-type spaces}, denoted $D_\alpha$ ($\alpha \in \R$). We consider this result in the case $\alpha=0$ (corresponding to the classical Hardy space $H^2$) to be known to experts. These spaces are appreciated by researchers in complex analysis and operator theory for the diversity of their features despite their one-parameter description. More precisely, non-cyclicity is preserved only in the trivial case $\alpha>1$, where a simple characterization of cyclic functions is available. These results also seem to be folklore (indeed, some of them are implicit in \cite{BS84}), and their proofs are mostly simple, but our personal experience tells us that the conclusions obtained are contrary to what some audiences seem to expect.

An approach for the study of cyclicity based on the so called \emph{optimal polynomial approximants} (or o.p.a.) associated to a function $f$ was developed in \cite{BCLSS15}. To each function $f \in D_\alpha$, one associates a sequence of polynomials $\{p_d: d  \in \N \}$, each $p_d$ of the corresponding degree $d$. Suppose $\{f_n\}$ is a sequence of elements of $\mathcal{H}$ converging to $f$ in the norm of the space. Although the cyclicity of $\{f_n\}$ do not determine that $f$ has the same property, we will show that some features related to the study of o.p.a. are preserved in some way by convergence. 

One such specific feature arises from two basic properties of o.p.a.: one is that the o.p.a. of all degrees associated to each function $f$, together with the value $f(0)$, characterize the cyclicity of $f$; the other property is that each of the o.p.a. is determined by a corresponding linear system. If we denote by $p_{d}$ the o.p.a. of degree $d$ associated to $f$ and by $M$, $c$, $b$ the matrix, unknown and independent term of the corresponding linear system, and for each $n\in\N$ we denote by $p_{n,d}$, $M_n$, $c_n$, $b_n$,  the analogues for $f_n$, we will prove that $M_n$ converges to $M$, $c_n$ to $c$, $b_n$ to $b$, and $p_{n,d}$ to $p_{d}$, all of them, at a speed \emph{in terms of $n$} that is comparable to the rate at which $\|f_n-f\|_{\mathcal{H}}$ converges to 0, although the comparison constants $C(d,\alpha,f)$.

If one finds uniform bounds for $C(d, \alpha, f)$ of the form $C'(d, \alpha)$, it is possible to extrapolate cyclic properties to the limit function. However we were able to show that, unfortunately, the constants of comparison are different for each $f$ and are not uniformly bounded.

We will begin in Section \ref{Sect2} by presenting some preliminaries on the spaces we will study, making more precise the definitions needed. Then, in Section \ref{Sect3}, we will provide some very strong assumptions under which a sequence of cyclic functions do indeed converge to a cyclic function, as well as the folklore result, which consist of  counterexamples on how the additional hypothesis cannot be ignored for the preservation of cyclicity. This is presented in Theorem \ref{cyclic_limit_alpha>1} below. Then, we do a similar study of the question on the preservation of non-cyclicity, in Proposition \ref{thm32}, but here the answer depends on the space we look at. We will dedicate Section \ref{sect4} to the study of o.p.a. and their convergence. With the notation above, our main result is Theorem \ref{thm1a} below, which gives quantitative estimates on the norms of $b_n-b$, $M_n-M$ and $p_{n,d}-p_d$ in terms of those of $f_n-f$. Then we clarify how the corresponding constants (for fixed $f$) may blow up as $f$ moves through a Dirichlet-type space. This is done in Proposition \ref{propo45}.

	\section{Preliminaries}\label{Sect2}
	\subsection{General properties of Dirichlet-type spaces}
	For $\alpha \in \R$, we consider the Hilbert space $D_{\alpha}$ of all analytic functions $f$ on $\D$ satisfying
	$$\|f\|^2_\alpha := \sum_{k=0}^{\infty} |a_{k}|^{2} (k+1)^{\alpha} < \infty$$
	where $f(z)= \sum_{k=0}^{\infty} a_{k} z^{k}$, $z\in\D$.
	The inner product in $D_{\alpha}$ is defined by
	$$\langle f,g\rangle_{\alpha} := \sum_{k=0}^{\infty} a_{k} \overline{b_{k}} (k+1)^{\alpha}$$
	with
	$g(z) = \sum_{k=0}^{\infty} b_{k} z^{k}$. These are called \emph{Dirichlet-type spaces} and from their definition, it is easily seen that $D_{\alpha}\subseteq D_{\beta}$ for $\alpha \geq \beta$ and that $f \in D_\alpha$ if and only if $f' \in D_{\alpha-2}$, but apart from that, these spaces have very different behavior in many aspects, such as their boundary behavior.
	
	In the cases when $\alpha=-1,0,1$, $D_{\alpha}$ correspond, respectively, to the Bergman $A^{2}$, Hardy $H^{2}$ and Dirichlet $D$ spaces. We refer the reader to \cite{DurSchu, Ga81, EKMR14} for more on these classical and widely studied objects. 
	For $\alpha>1$, $D_{\alpha}$ is a subalgebra of the disk algebra, that is, the set $\mathcal{A}(\D)$ of all functions holomorphic in $\D$ and continuous on $\overline{\D}$ . 
		
	A function $\phi \in D_\alpha$ is called a {\it multiplier} (on $D_{\alpha}$) if $\phi D_{\alpha} \subset D_{\alpha}$, that is, if the {\it multiplication operator}  $M_{\phi}: D_{\alpha} \to D_{\alpha}$ given by $M_{\phi}f= \phi f$ is well-defined and bounded (see \cite{Ta66}). The set of all multipliers on $D_{\alpha}$ is denoted by $M(D_{\alpha})$. We have that for $\alpha \leq 0$, $M(D_{\alpha}) = H^{\infty}$, and for $\alpha> 1$, $M(D_{\alpha})=D_{\alpha}$. Since $1\in D_{\alpha}$, any multiplier of $D_{\alpha}$ is an element of $D_{\alpha}$. We define a norm on the multiplier space, by assigning $\phi$ the operator norm of $M_{\phi}$, which we denote $\|\phi\|_{M(D_\alpha)}$.  In this way $M(D_\alpha)$ becomes a Banach space.
	
	The {\it shift operator} $S:D_{\alpha} \rightarrow D_{\alpha}$ is the bounded operator defined by $Sf(z) = zf(z)$, i.e., $S$ is the multiplication operator associated to the function $\phi(z)=z$.
	\begin{rem}\label{rem001}
 	Observe from the definition of the spaces that if $\alpha =0$, then $S$ is an isometry; for $\alpha < 0$, $S$ is a contractive operator with norm equal to $1$; and for $\alpha>0$ one has $\|S\| = 2^{\alpha/2}$.
	\end{rem}
 	
 	We are interested in discussing some aspects about cyclic functions for the shift operator in $D_{\alpha}$. We recall that given a Hilbert space $\HH$ and a linear bounded operator $T:\HH\to \HH$, a vector $v\in \HH$ is said to be {\it cyclic} for $T$ if and only if
 	\[\HH = \overline{\Span\{T^{k}v:k\in\N\}}. \]

Notice that Definition \ref{def1}	simply says that $f$ is cyclic if it is cyclic \emph{for the shift operator}.

It is well known that $D_{\alpha}$ are examples of \emph{reproducing kernel Hilbert spaces} with kernel
	$$k_{w}(z) = \sum_{k=0}^{\infty} \frac{\overline{w}^{k} z^{k}}{(k+1)^{\alpha}}.$$ 
	This is to say that  for all $w\in \D$ the function  $k_w$ has the special property that $f(w) = \langle f, k_{w} \rangle_{\alpha}$ for any $f \in D_\alpha$.
	
	\begin{rem} 
	   Observe that since $D_{\alpha}$ is a reproducing kernel Hilbert space, the point evaluation functionals are bounded. Also, one can see from the definition of the spaces and \eqref{eqn2001} that if a function $f$ is cyclic for $D_{\alpha}$, then $f$ is cyclic for $D_{\beta}$ for all $\beta \leq \alpha$.
	\end{rem}

	 When $\alpha>1$, any function in $D_\alpha$ is automatically continuous to the boundary $\T$ and this allows one to show the following complete characterization of cyclic functions available, due to Brown-Shields \cite{BS84}:
	 
	\begin{prop}\label{cyclic_alpha>1}
		For $\alpha>1$, $f$ is cyclic for $D_{\alpha}$ if and only if $f$ has no zeros in the closed unit disc, that is, if there exists a constant $c$ such that
		\begin{equation}\label{cond_ciclycity}
		   |f(z)| > c > 0, \qquad \text{ for all } z\in\mathbb{D}.
		\end{equation}
	\end{prop}
	
	\begin{rem}\label{cyclicity}
	For $\alpha \leq 1$,  boundary values are not necessarily well defined but if they are, then \eqref{cond_ciclycity} is sufficient for $f$ to be cyclic for $D_{\alpha}$ (although not necessary). A necessary but not sufficient condition is that $f$ has no zeros in $\D$. 
	\end{rem}

\subsection{Carleson sets, capacity and cyclic functions}\label{carleson_capacity}
	We now present some families of geometric measure theory concepts, strongly tied to $D_\alpha$ spaces. For a wider exposition on these topics we refer to \cite{EKMR14}. 
	
	A \emph{Carleson set} is a closed subset $E$ of $\T$ such that
	$$\int_{\T} \log\left( \frac{1}{{\rm dist}(\zeta,E)} \right)\, |d\zeta| < \infty.$$
	
	Equivalently, a set $E\subset \T$ is a Carleson set if and only if $|E|=0$ and 
	\begin{equation}\label{carleson_set}
	    \sum_{k}|I_{k}| \log(1/|I_{k}|)<\infty,
	\end{equation}
	where $\{I_{k}\}$ are the connected components of $\T \setminus E$ (each $I_k$ being an interval).
	
Let $K_\alpha: (0,\infty)\to [0,\infty)$ be the function given by $K_\alpha(t) = t^{\alpha-1}$, when $0 < \alpha <1$. We refer to $K_\alpha$ as the Riesz kernel. For $\alpha =1$, $K_1(t) = \max\{\log(2/t),0\}$, is the \emph{logarithmic kernel}. For a compact subset $E$ of $\T$, we define its \emph{$\alpha$-capacity} $c_\alpha$ ($\alpha$-Riesz capacity or logarithmic capacity, respectively) as
\[c_{\alpha}(E) = 1/ \inf\left\{\int_{\T} \int_{\T} K_{\alpha}(|x-y|)\,d\mu(x)\,d\mu(y)\right\},\]
    where the infimum is taken over the set of all Borel probability measures on $E$.
This concept satisfies some basic properties as mapping the empty set to zero, being monotone and finitely subaditive \cite[Theorem 2.1.3]{EKMR14}. In order to apply later some results on $c_{\alpha}$ from \cite{EKMR14}, we remark that our $D_\alpha$ spaces appear as $D_{1-\alpha}$ in the notation there (since we focus on the description of spaces in terms of coefficients rather than in terms of measures).

For our purposes, we say that $E$ is a \emph{Cantor set} if it is a compact subset of the unit circle $\T$ defined through its components as follows: if $\{E_{n}^{j} : 1\leq j \leq 2^{n},\,n\geq 0\}$ are non-empty compact subsets of $\T$, such that for $j=1,...,2^{n}$ the sets $E_{n}^{j}$ are pairwise disjoints and each set $E_{n-1}^{i}$ contains precisely two sets $E_{n}^{j}$, then \[E=\bigcup_{n=0}^{\infty} \bigcap_{j=1}^{2^{n}} E_{n}^{j}.\]

With this in mind, we can write \cite[Theorem 2.3.5]{EKMR14} as follows:
\begin{thm}\label{positive_capacity}
    Let $E$ the cantor set corresponding to $\{E_{n}^{j}\}$. Then,
    \begin{equation}\label{eqn2002}
        \sum_{n=0}^{\infty} \frac{K_\alpha(d_{n})}{2^{n+1}} 
        \leq \frac{1}{c_{\alpha}(E)} 
        \leq \sum_{n=0}^{\infty} \frac{K_\alpha(e_{n})}{2^{n+1}},
        \end{equation}
        where
        \begin{align*}
         d_{n}&=\max\{{\rm diam}(E_{n}^{j}) : 1\leq j\leq 2^{n}\} \\
         e_{n} &= \min\{d(E_{n+1}^{j},E_{n+1}^{k}): E_{n+1}^{j},E_{n+1}^{k}\subset E_{n}^{i} \text{ for } j\neq k\}.
        \end{align*}
        In particular, if the right-hand side in \eqref{eqn2002} converges, then $c_{\alpha}(E)>0$. 
    \end{thm}

    \section{Qualitative results}\label{Sect3}
	From now on, we denote by $Z(g)$ the set of zeros of a function $g$ in its domain. Given a sequence $\{f_{n}\}_{n=1}^{\infty} \subset D_{\alpha}$ of cyclic functions that converges to a non-zero function $f\in D_{\alpha}$, we investigate if $f$ must be cyclic too. The following theorem gives an affirmative answer under additional hypothesis and counterexamples otherwise.
	
	\begin{thm}\label{cyclic_limit_alpha>1}
		Let $\alpha\in \R$ and let $\{f_{n}\}_{n=1}^{\infty}$ be a sequence of cyclic functions for $D_{\alpha}$ converging to $f$ in the norm of $D_{\alpha}$. Assume that the following conditions hold:
		\begin{enumerate}
			\item\label{f_multiplier} there exists $C>0$ such that $\|f/f_{n}\|_{M(D_{\alpha})} \leq C$ for all $n\in\N$,
			\item\label{1/fn_bounded} $\|1-f/f_{n}\|_{\alpha} \rightarrow 0$ as $n$ tends to $\infty$. 
		\end{enumerate}
		Then, $f$ is cyclic in $D_{\alpha}$.
		
		Moreover, for each $\alpha \in \R$, there exists a non-cyclic function $f \in D_\alpha$ such that $Z(f)\cap \D = \emptyset$ and a sequence of cyclic functions $\{f_n\}$ converging to $f$ in $D_\alpha$, satisfying (b) and with each $f/f_n$ being a multiplier.
	\end{thm}

	\begin{proof}
		Let $\epsilon>0$. Since $f_{n}\in D_{\alpha}$ is cyclic for every $n\in \N$, there exists a sequence of polynomials $\{p_{m,n}\}_{m=0}^{\infty}$ and a constant $M(n)$ depending on $n$ such that $\|p_{m,n} f_{n} - 1\|_{\alpha} < \epsilon$ if $m>M(n)$.  
		We will see that we can choose $n_{0}\in \N$ such that $\|p_{m,n_{0}} f - 1\|_{\alpha} < \epsilon$ if $m>M(n_{0})$. We can easily decompose $p_{m,n_0}f-1$ into a difference yielding:
\[			\|p_{m,n_{0}} f - 1\|_{\alpha}
			 = \|f/f_{n_{0}} (p_{m,n_{0}} f_{n_{0}} -1) - (1- f/f_{n_{0}}) \|_{\alpha}. \]
The definition of the multiplier norm and the triangle inequality guarantee that this is bounded above by
\[ \|f/f_{n_{0}}\|_{M(D_{\alpha})} \|p_{m,n_{0}} f_{n_{0}} -1 \|_{\alpha} +  \| 1 - f/f_{n_{0}}\|_{\alpha}. \]
Our first assumption guarantees now that the first term is controlled by   \[ C \|p_{m,n_{0}} f_{n_{0}} -1 \|_{\alpha} < C\epsilon \] while the second hypothesis tells us that the second term is also arbitrarily small. This proves the main conclusion of the theorem.
	
	To prove the remaining part, we separate the cases $\alpha \leq 0$, $0<\alpha \leq 1$ and $\alpha>1$. 
	
	For $\alpha \leq 0$, consider the singular inner function 
	\begin{equation}
	    g(z)=e^{-\frac{1+z}{1-z}}.
	\end{equation}
	which verifies that $|g(z)|<1$ for all $z\in \D$. Then, $g\in H^{\infty}=M(D_{\alpha}) \subset D_{\alpha }$ for all $\alpha \leq 0$. Moreover, $g$ is not cyclic for $D_{\alpha}$ (see for example \cite{Ko81}). 
	
	We shall see that $g$ can be approximated in $D_{\alpha}$ by a sequence $\{g_{n}\}_{n\in\N}$ of cyclic functions. To prove this, consider for $n\in \N$  $$g_{n}(z)=e^{-\frac{1+z}{(1+1/n)-z}}$$ and observe that $g_{n}$ satisfies the following inequality
	\begin{equation}
	e^{-2n} \leq |g_{n}(z)| 
	< 1, \qquad z\in \D.
	\end{equation}
	Hence, $g_{n} \in H^{\infty} \subset D_{\alpha}$ and it is bounded below by a positive constant for each $n\in \N$ which implies, by Remark \ref{cyclicity}, that $g_{n}$ is cyclic for all $n\in \N$. 
	
	It remains to be seen that $\|g_{n}-g\|_{\alpha}\to 0$ as $n\to\infty$. Since $\|f\|_{\alpha} \leq \|f\|_{0}$ for all $f\in D_{\alpha}$ and for all $\alpha\leq 0$, it suffices to show the convergence in $H^2$. Now, notice first that since $g_{n}$ is holomorphic across the boundary $\T$, it has a radial limit $g_{n}^{*}(e^{i\theta})=g_{n}(e^{i\theta}) $ for all $\theta \in (-\pi,\pi]$. On the other hand, if $\theta\neq 0$, the limit function $g$ has radial limit $g^{*}(e^{i\theta}) = g(e^{i\theta}) = -i e^{ \cot(\theta/2)}$ and $g^{*}(1)=0$. Hence, we can decompose
	\[\|g_{n}-g\|_{0}^{2}
		= \frac{1}{2\pi}\int_{-\pi}^{\pi} |g^{*}_{n}(e^{i\theta}) - g^{*}(e^{i\theta})|^{2}\,d\theta = (I) + (II), \]
where $(I)$ is the integral of the same quantity over $\epsilon \leq |\theta|\leq \pi$ and $(II)$ that over the remaining cases $|\theta| < \epsilon$.

From the definition of $g_{n}$ and the fact that $|g(z)|\leq 1$ for all $z\in\D$, it is standard to see that \[|g^{*}_{n}(e^{i\theta}) - g^{*}(e^{i\theta})|^{2} \leq 4.\]
	For $\epsilon\leq |\theta|\leq \pi$, we can moreover see that 
	\begin{equation}\label{g_convergence}
	    |g^{*}_{n}(e^{i\theta}) - g^{*}(e^{i\theta})|^{2}
	    = e^{2\Re(A)} - 2 \Re(e^{A})+1
	\end{equation}
	where $$A= \frac{1}{n} \frac{1+e^{i\theta}}{((1+1/n)-z)(1-e^{i\theta})},$$ which converges to $0$ as $n$ tends to $\infty$. So, \eqref{g_convergence} converges to $0$ as $n\rightarrow \infty$. Then, by Lebesgue's Dominated Convergence Theorem, (I) tends to zero when $n\rightarrow \infty$. The integral (II) is less than $4\epsilon/\pi$. Since this holds for every $\epsilon>0$, we get that $\|g_{n}-g\|_{0}\to 0$ when $n\rightarrow \infty$.

	The case $0< \alpha \leq 1$ is a little more subtle. Let $E\subset \T$ be a set with Lebesgue measure zero and positive $\alpha$-capacity. This set can be constructed as follows: Fix $\alpha$ and take $\varepsilon=1-2^{\alpha/(\alpha-2)}$. Observe that then 
	\begin{equation}
	    \frac{\alpha}{2}= 1- \frac{ \log(2)}{\log(\frac{2}{1-\varepsilon})}.
	\end{equation}
	
	Consider the Cantor set $E$ obtained by removing in each step an open middle-interval of length $\varepsilon$ times the length of each interval on the level $k-1$. It is a standard computation to see that $E$ has Lebesgue measure zero and Hausdorff dimension $\log(2)/\log(2/1-\varepsilon)$. 

	Therefore,  (since $\alpha> 1 - \log(2)/\log(2/1-\varepsilon)$) $c_{\alpha}(E)>0$ by Theorem \ref{positive_capacity}.  Moreover, we claim that $E$ is a Carleson set and this is easily seen using the condition \eqref{carleson_set}.
	
	Now, there exists an outer non-cyclic multiplier $f\in D\cap \mathcal{C}(\overline{\D})$ (and consequently $f\in D_{\alpha}\cap \mathcal{C}(\overline{\D})$ for $0<\alpha<1$) such that its zero set coincides with $E$ (see \cite[Theorem 9.2.8]{EKMR14}). In addition, the restriction of $f$ to $\T$ defines a \emph{non-negative} real function. 
	
	The idea is to define a sequence of cyclic functions that converges to $f$. This is done by adding the constant $1/n$ for $n\in \N$ to $f$, obtaining a sequence  $\{f_{n} = f+1/n\}_{n=1}^{\infty} \subset D_{\alpha}$. Since all of the functions $f_n$ are outer and their restrictions to $\T$ are bounded below by $1/n$ for the corresponding $n\in\N$, they are all cyclic in $D_\alpha$.

	Finally, for $\alpha>1$ and $n\in \N$, define $f_{n}(z)= 1 + \frac{1}{n}-z$. Then $\{f_{n}\}_{n=1}^{\infty}$ is a sequence of cyclic functions for $D_{\alpha}$ (by Proposition \ref{cyclic_alpha>1}) that converges to $f(z)=1-z$. However, from the same proposition it is also clear that $f$ is not cyclic for $D_{\alpha}$ for $\alpha >1$.
	\end{proof}

The previous result has the following consequence:

\begin{corollary}
The set of cyclic functions in $D_\alpha$ is not closed in the topology of $D_\alpha$. In particular, the set of outer functions is not closed in the topology of $H^2$.
\end{corollary}	

For the preservation of non-cyclicity we can also construct counterexamples but only when $\alpha \leq 1$:

	\begin{prop}\label{thm32} Let $\alpha \in \R$. Then:
	\begin{enumerate}
		\item If $\alpha>1$, the limit of a convergent sequence of non-cyclic functions for $D_{\alpha}$ is not cyclic.
		
		\item If $\alpha \leq 1$, there exists a sequence of non-cyclic functions for $D_{\alpha}$ that converges to a cyclic one.
	\end{enumerate}

	\end{prop}

	\begin{proof}
		(a) Let $\alpha>1$ and let $\{f_{n}\}_{n=1}^{\infty}$ be a sequence of non-cyclic functions for $D_{\alpha}$ converging to $f\in D_{\alpha}$. Since $f_{n}$ is not cyclic, Proposition \ref{cyclic_alpha>1} implies that there is $z_{n}\in\overline{\D}$ such that $f_{n}(z_{n})=0$. By compactness, there is a convergent subsequence $\{z_{n_k}\}_{k=1}^{\infty}$ of $\{z_{n}\}_{n=1}^{\infty}$ and an accumulation point of that subsequence $z_{\infty}\in \overline{\D}$. It must be the case that $f(z_\infty)=0$, and thus $f$ is not cyclic, by applying again Proposition \ref{cyclic_alpha>1}.

		(b) Consider $f_{n}(z)= 1-\frac{1}{n}-z$ and $f(z)=1-z$. Then, $f_{n}, f\in D_{\alpha}$ (since they are polynomials), $\{f_{n}\}_{n=1}^{\infty}$ converges to $f$, $f_{n}$ is not cyclic for $D_{\alpha}$ for $n\in \N$ since its only zero is inside $\D$, but $f$ is cyclic for $D_{\alpha}$. The latter follows from the fact that $1-z$ is cyclic in $D_1$  \cite[Lemma 8]{BS84} and the fact that if a function $g$ is cyclic for $D_{1}$ then it is cyclic for $D_{\alpha}$ for $\alpha \leq 1$, which is clear from the equivalent form of cyclicity in \eqref{eqn2001}. 
		\end{proof}

\section{Quantitative results}\label{sect4}

Denote by $\mathcal{P}_{d}$ the space of all polynomials of degree at most $d$ equipped with the norm $\|\cdot\|_{\alpha}$. Then, for $f \in D_\alpha\backslash \{0\}$ the orthogonal projection of $1$ onto the finite-dimensional Hilbert space $f\mathcal{P}_{d} = \{pf : p \in \mathcal{P}_{d} \}$ always exists and is uniquely determined, yielding an element $p_{d}^{*}f$. The element $p_{d}^{*} \in \mathcal{P}_{d}$ minimizes the norm $\|pf-1\|_{\alpha}$ and it is called the {\it $d$-th optimal polynomial approximant (o.p.a.)} to $1/f$. From \eqref{eqn2001} it is clear now, that a function $f$ is cyclic if and only if \[\|p_d^* f-1\|_\alpha \rightarrow 0 \text{ as  } d\rightarrow \infty.\]
	    
In \cite[Theorem 2.1]{FMS14}, the authors proved that the coefficients of the $d$-th o.p.a. $p_{d}^{*}$ to $1/f$ are the entries of the only solution $c\in \C^{d+1}$ to the linear system
	    \begin{equation}\label{linear_system_opa}
	          Mc=b
	    \end{equation}
	    where $M\in \C^{(d+1)\times(d+1)}$ is the matrix with entries
	    $$m_{j,k} = \langle S^{k}f,S^{j}f \rangle_{\alpha} \qquad \text{ for } \qquad  j,k=0,...,d,$$
	    and $b \in \C^{(d+1)}$ is the vector with entries \[b_{k}=\langle 1, S^{k}f \rangle_{\alpha}  \qquad \text{ for } \qquad  k=0,...,d .\]
	    
	    Given $f\in D_{\alpha}\setminus\{0\}$ and $\{f_{n}\}_{n=1}^{\infty}\subset D_{\alpha}$ which converges to $f$, then the $d$-th o.p.a. to $1/f$ can be approximated by the o.p.a. to $1/f_{n}$. We are ready to see a quantitative estimate of the distance between the $d$-th o.p.a. corresponding to $f$ and $f_n$. Denote by $\|\cdot\|_F$ the Frobenius norm for matrices.
 		
 		\begin{thm}\label{thm1a}
 			Let $\alpha\in \R$, $d\in\N$ and $\{f_{n}\}_{n=1}^{\infty}$ be a sequence in $D_{\alpha}$ and $f\in  D_{\alpha}$. Consider the $d$-th o.p.a. $p_{d}$ to $1/f$ and the $d$-th o.p.a. $p_{d,n}$ to $1/f_{n}$, and their corresponding linear systems of the form \eqref{linear_system_opa}, denoted by $Mc=b$ and $M_{n}c_{n}=b_{n}$, $n\in \N$, respectively. Then, there exist constants $\varphi(d,\alpha)$ independent of $n$ such that for all $n\in\N$	\begin{enumerate}\label{convergence_oap}
 			    \item[(a)] $\|b_{n}-b\|_{0} \leq \|f_{n} - f\|_{\alpha}$,
 			    \item[(b)] $\|M_{n}-M\|_{F}
 			    \leq \varphi(d,\alpha) \,(\|f_{n}\|_{\alpha} + \|f\|_{\alpha}) \, \|f_{n}-f\|_{\alpha}$.
 			    
 			\end{enumerate}
 			
 			Moreover, there exists $\psi(d,\alpha)$ such that
 		\[\|p_{d,n} - p_{d}\|_{\alpha} \leq \psi(d,\alpha) \|M_{n}^{-1}\|_{F} \left( \|b_{n}-b\|_{0} + \|M^{-1}\|_{F} \|M-M_{n}\|_{F}  \|b\|_{0}\right).\]
 		\end{thm}
 		
Naturally, we are more interested in the case in which the sequence $f_n$ is converging in $D_\alpha$ to $f$. In that case, we have the following straightforward implication.
 		
 	    \begin{corollary}
 	Under the same hypotesis as in Theorem \ref{convergence_oap}, if we assume that $\{f_{n}\}_{n=1}^{\infty}$ converges to $f$ in the norm of $D_{\alpha}$, then
 	\begin{equation}
 	    \lim_{n\rightarrow \infty} \|b_{n} - b\|_{0}= \lim_{n\rightarrow \infty} \|M_{n} - M\|_{F}= \lim_{n\rightarrow \infty}\|p_{d,n} - p_{d}\|_{\alpha}=0
 	\end{equation}
with a rate of decay in all cases of the form  $C(d,\alpha,f)\|f_{n} - f\|_{\alpha}$.
 		\end{corollary}

    	\begin{rem}\label{remark_linear_trans}
 		Before we prove Theorem \ref{convergence_oap} we want to discuss some aspects about linear transformations $T:(\mathcal{P}_{d},\|\cdot\|_{0}) \to (\mathcal{P}_{d},\|\cdot\|_{0})$. Let $\cB= \{z^{k}\}_{k=0}^{d}$ be the standard orthonormal basis of $(\mathcal{P}_{d},\|\cdot\|_{0})$.
 		Then, if $[T]_{\cB}= \left(t_{j,k}\right)_{j,k=0}^d$ is the corresponding matrix and $[p]_{\cB}$ is the coordinate vector of $p\in \mathcal{P}_{d}$ in the same basis, we have that
 		$[Tp]_{\cB} = [T]_{\cB}[p]_{\cB}$. Moreover, by Cauchy-Schwarz inequality 
 		\begin{align*}
 		    \|T\|^{2} := \sup_{\|p\|_{0}=1} \|[Tp]_{\cB}\|^{2}_{0} 
 		    = \sup_{\|p\|_{0}=1} \sum_{j=0}^{d} \left| \sum_{k=0}^{d} t_{j,k} a_{k} \right|^{2}
 		    & \leq  \sum_{j=0}^{d} \left( \sum_{k=0}^{d} |t_{j,k}|^{2} \right).
 		\end{align*}
 		The right-hand side above is equal to $\|[Tp]_{\cB}\|^{2}_{F}$.

 		Our goal is to determine the matrix representation of $T$ being defined in the space $\mathcal{P}_{d}$ endowed with the norm $\|\cdot\|_{\alpha}$. Let $\cB^{'}= \{z^{k}/(k+1)^{\alpha/2}\}_{k=0}^{d}$ be the orthonormal basis of $(\mathcal{P}_{d},\|\cdot\|_{\alpha})$.
 		The change of basis matrix from $\cB^{'}$ to $\cB$ is 
 		$$D = \begin{pmatrix}
 		1 & 0 & \dots & 0 \\
 		0 & 1/2^{\alpha/2} & \dots & 0 \\
 		\vdots & \vdots & \ddots & \vdots \\
 		0 & 0 & \dots & 1/(d+1)^{\alpha/2}
 		\end{pmatrix},$$
 		and then, $D [p]_{\cB^{'}} = [p]_{\cB}$. It follows that
 		$$ [T]_{\cB^{'}} [p]_{\cB^{'}} =  [Tp]_{\cB^{'}} = D^{-1} [Tp]_{\cB} = D^{-1} [T]_{\cB} [p]_{\cB} = D^{-1} [T]_{\cB} D [p]_{\cB^{'}},$$
 		that is,
 		$$[T]_{\cB^{'}} = D^{-1} [T]_{\cB} D.$$
 		As a consequence of the above, observe that
 \begin{equation}\label{eqn2003}
\|[T]_{\cB^{'}}\|_{F}^{2} = {\rm trace}([T]_{\cB^{'}}^{*} [T]_{\cB^{'}}) = \sum_{j,k=0}^{m} |t_{j,k}|^{2} \frac{(j+1)^{\alpha}}{(k+1)^{\alpha}}.\end{equation}
 		
 		\end{rem}
 		
 		\begin{proof}[Proof of Theorem \ref{convergence_oap}]
 		
        Property (a) is immediate once we notice that $|b_{n,0} - b_{0}|=|f_{n}(0) - f(0)|$ and $|b_{n,k} - b_{k}|=0$ for all $k\geq 1$. Hence,
 	\[\|b_{n}-b\|_{0}^{2} = \sum_{k=0}^{d} |b_{n,k}- b_{k}|^{2} = |f_{n}(0) - f(0)|^{2} \leq \|f_{n}-f\|^{2}_{\alpha}.\]

 		Now, we turn to show the validity of (b). We write $m_{j,k}^{n}$ and $m_{j,k}$ to denote the entries of $M_{n}$ and $M$ respectively. Indeed, from their definition, we see that
 \[   |m_{j,k}^{n}-m_{j,k}|
 			    = | \langle S^{k}f_{n},S^{j}f_{n} \rangle_{\alpha} - \langle S^{k}f,S^{j}f \rangle_{\alpha} |. \]
The triangle inequality yields that
 \[   |m_{j,k}^{n}-m_{j,k}|
\leq |\langle S^{k} f_{n} , S^{j}(f_{n}-f) \rangle_{\alpha}| + |\langle S^{k} (f_{n}-f) ,  S^{j} f \rangle_{\alpha}|. \]
By the definition of the operator norm and the Cauchy-Schwarz inequality, one can isolate the dependence on $S^j$ and get:
 \[   |m_{j,k}^{n}-m_{j,k}| \leq \|S^{j}\| \|S^{k}\| \|f_{n}-f\|_{\alpha} \left( \|f_{n}\|_{\alpha} + \|f\|_{\alpha} \right).\]

		Note that $\|S^{t}\|$ is comparable to $(t+1)^{\alpha/2}$ when $\alpha \geq 0$ and to $1$ otherwise. Then,
\[\|M_{n}-M\|_{F}^{2} = \sum_{j,k=0}^{d} |m_{j,k}^{n} - m_{j,k}|^{2} \frac{(j+1)^{\alpha}}{(k+1)^{\alpha}}\] may be bounded from above by
\[ \varphi(d,\alpha)^{2} (\|f_{n}\|_{\alpha} + \|f\|_{\alpha})^{2}\|f_{n}-f\|_{\alpha}^{2},\] where, using the notation in Remark \ref{remark_linear_trans},  \begin{equation}\label{expression_phi}
 		        \varphi(d,\alpha) =  \begin{cases} 
 			    \|D\|_{F} \|D^{-1}\|_{F}& \text{ if } \alpha < 0, \\
 			    d+1 & \text{ if } \alpha=0, \\
 			    (d+1) \|D^{-2}\|_{F}& \text{ if } \alpha>0.
 			    \end{cases}
 		    \end{equation}

 		 This proves (b).
 		 
 		Finally, by using (a) and (b) we are going to prove that $\|p_{d,n}-p_{d}\|_{\alpha}$ is bounded and describe its dependence on $\|f_{n}-f\|_{\alpha}$.
 		We express the information we have in the notation of Remark \ref{remark_linear_trans}. If we identify $b_n$ and $b$ with the constants $b_n(0)$ and $b(0)$ as elements of $\mathcal{P}_d$, we can write the equations in \eqref{linear_system_opa} as
 		\begin{equation}\label{n_linear_systems}
 		    M_{n}[p_{n,d}]_{\cB} = [b_{n}]_{\cB} \quad \text{ and } \quad  M[p_{d}]_{\cB} = [b]_{\cB}.  
 		\end{equation}
 		
 		Observe that \eqref{n_linear_systems} holds with respect to any basis, in particular, in the basis $\cB^{'}$ that we define in Remark \ref{remark_linear_trans}.
 		
 		Then we get that
 \[   \|p_{d,n} - p_{d}\|_{\alpha}^{2} = \sum_{k=0}^{d} |c_{k}^{n}(k+1)^{\alpha/2} - c_{k}(k+1)^{\alpha/2}|^{2}\]	which is equal to \[\left\|[p_{d,n}]_{\cB^{'}}-[p_{d}]_{\cB^{'}} \right\|_{0}^{2}
 		   = \left\|[M_{n}]_{\cB^{'}}^{-1} [b_{n}]_{\cB^{'}} - [M]_{\cB^{'}}^{-1} [b]_{\cB^{'}} \right\|_{0}^{2}. \]
 		Applying the triangle inequality, we have that $ \|p_{d,n} - p_{d}\|_{\alpha}$ is bounded above by \[\left\|[M_{n}]_{\cB^{'}}^{-1} [b_n]_{\cB^{'}} - [M_{n}]_{\cB^{'}}^{-1} [b]_{\cB^{'}} \right\|_{0} +  \left\| [M_{n}]_{\cB^{'}}^{-1} [b]_{\cB^{'}} - [M]_{\cB^{'}}^{-1} [b]_{\cB^{'}}  \right\|_{0}. \]
  Recalling that $[M_{n}]_{\cB^{'}}^{-1} = D^{-1}M_{n}^{-1}D$, $[M]_{\cB^{'}}^{-1} = D^{-1}M^{-1}D$ and $[b]_{\cB^{'}}=D^{-1}[b]_{\cB}$ we have the following estimate for the first term on the right-hand side:
 \[ \left\|[M_{n}]_{\cB^{'}}^{-1} [b_n]_{\cB^{'}} - [M_{n}]_{\cB^{'}}^{-1} [b]_{\cB^{'}} \right\|_{0} \leq \|D^{-1}\|_{F}\|M_{n}^{-1}\|_{F} \|b_{n}-b\|_{0}.\]
  In order to control the second term as well, we use the following factorization: if $A,B$ are two invertible matrices, then \[A^{-1}-B^{-1} = B^{-1}(B-A) A^{-1},\]
 		 and therefore \[\|A^{-1}-B^{-1}\|_{F} \leq \|B^{-1}\|_{F} \|B-A\|_{F} \|A^{-1}\|_{F}.\] This last inequality implies that \[  \left\| [M_{n}]_{\cB^{'}}^{-1} [b]_{\cB^{'}} - [M]_{\cB^{'}}^{-1} [b]_{\cB^{'}}  \right\|_{0}  \leq \|D^{-1}\|_{F}
 		 \left\| M_{n}^{-1}  - M^{-1} \right\|_{F} \left\| b \right\|_{0}, \]
 which is bounded above by \[\|D^{-1}\|_{F} \|M^{-1}\|_{F} \left\|M - M_{n} \right\|_{F} \|M_{n}^{-1}\|_{F}  \left\| b \right\|_{0}.\]
 Summarizing we obtained that $\|p_{d,n} - p_{d}\|_{\alpha}$ is bounded by \[\|D^{-1}\|_{F} \|M_{n}^{-1}\|_{F} \left( \|b_{n}-b\|_{0} + \|M^{-1}\|_{F} \|M-M_{n}\|_{F}  \|b\|_{0}\right). \]
Noticing that $\|D^{-1}\|_{F}$ only depends on $d$ and $\alpha$, the result follows by taking $\psi(d,\alpha) = \|D^{-1}\|_{F}$.
 		\end{proof}

One of our main interests has been to find explicit estimates, which are constructed in Theorem \ref{thm1a}. The following corollary of the proof of the theorem specifies the nature of the constants.
 	\begin{corollary}
 	The constants $\varphi(d,\alpha)$ and $\psi(d,\alpha)$ in Theorem \ref{convergence_oap} can be taken to be $\varphi(d,\alpha) = C(\alpha) \varphi'(d,\alpha)$ where
 	\begin{equation}
 	    \varphi'(d,\alpha) \leq   \begin{cases}
 	    (d+1)^{(2-\alpha)/2} & \text{ if } \alpha<-1, \\
 	    (d+1)^{(1-\alpha)/2} \sqrt{\log(d+2)} & \text{ if } \alpha = -1, \\
 		d+1 & \text{ if } -1<\alpha<0, \\
 		(d+1)^{\alpha+1} & \text{ if } \alpha \geq 0,
 		\end{cases}
 	\end{equation}
 	and 
 	\begin{equation}
 	    \psi(d,\alpha) \leq \begin{cases}
 			    d+1 & \text{ if } \alpha< -1, \\ \log(d+2)  & \text{ if } \alpha=-1, \\
 			    (\alpha+1)^{-1}(d+1)^{\alpha+1} & \text{ if }  \alpha > -1.
 			    \end{cases} 
 	\end{equation}
 	
 	\begin{proof}
 	
 	Recall that both constants depend only on the Frobenius norms of $D$, $D^{-1}$ and $D^{-2}$. So, estimating these norms for the different values of $\alpha$ we obtain 
 	\begin{equation}
 	\|D\|_{F}^{2} =
 	\sum_{k=0}^{d} \frac{1}{(k+1)^{\alpha}}
 	\leq (d+1)^{1-\alpha}, \quad \alpha \leq 0,
 	\end{equation}
 	and
 	\begin{equation}
 	\|D^{-1}\|_{F}^{2} =
 	\sum_{k=0}^{d} (k+1)^{\alpha}
 	\leq \begin{cases}
 			    d+1 & \text{ if } \alpha < -1, \\ \log(d+2)  & \text{ if } \alpha =-1, \\
 			    (\alpha+1)^{-1}(d+1)^{\alpha+1} & \text{ if }  \alpha > -1.
 			\end{cases}
 	\end{equation}
 	
 	Moreover, for $\alpha>0$ we have
 	\begin{equation}
 	    \|D^{-2}\|_{F}^{2} = \sum_{k=0}^{d} (k+1)^{2\alpha} \leq (d+1)^{2\alpha+1}.
 	\end{equation}
 	
 	Hence, using \eqref{expression_phi} we get the expected bounds.
 	    \end{proof}
 			
 		\end{corollary}
 		
 		Since the bound that we arrived to in Theorem \ref{convergence_oap} for the norm $\|p_{d,n} - p_{d}\|_{\alpha}$  depends on $\|M^{-1}\|_{F}$, we examine the question of whether there are uniform bounds for the values that may appear for different functions $f$ in a space $D_\alpha$ (for a normalized $f$, since the entries of $M$ are 2-homogeneous on $f$). It turns out that one can't make any improvements following that direction and we show this point in the following proposition.
 		
 		\begin{prop}\label{propo45}
 		Let $\alpha \in\R$. There exists a sequence of normalized functions $\{g_{n}\}_{n=0}^{\infty} \subset D_{\alpha} \setminus \{0\}$ with an o.p.a. of degree 1 to $1/g_{n}$ given by a sequence of $2\times 2$ matrices $\{M_n\}_{n=0}^{\infty}$ such that \[\|M_{n}^{-1}\|_{F} \to \infty \qquad \text{ as } \qquad n \to \infty.\]
 		\end{prop}
 		
 		\begin{proof}
 		 For each $n\in\N$, we explicitly write the matrix $M_{n}$ associated to the o.p.a. to $1/g_{n}$ given in the linear system \eqref{linear_system_opa}, that is,
 	    $$M_{n}= \begin{pmatrix}
 	    1 & \langle g_{n},Sg_{n}\rangle_{\alpha}  \\
 	    \langle Sg_{n},g_{n}\rangle_{\alpha} & \|Sg_{n}\|_{\alpha}^{2}
 	    \end{pmatrix}.$$
 	    
 	    Notice that $M_{n}^{-1}$ is well-defined since Cauchy-Schwarz inequality says that the determinant of $M_{n}$ is stricly positive. It is given by \[M_{n}^{-1} = \frac{1}{\|Sg_{n}\|_{\alpha}^{2} - | \langle Sg_{n},g_{n}\rangle_{\alpha}|^{2}} \begin{pmatrix}
 	     \|Sg_{n}\|_{\alpha}^{2} & -\langle g_{n},Sg_{n}\rangle_{\alpha}  \\
 	    -\langle Sg_{n},g_{n}\rangle_{\alpha} & 1
 	    \end{pmatrix},\]  whose Frobenius norm is \[ \|M_{n}^{-1}\|_{F}^{2} = \frac{1}{(\|Sg_{n}\|_{\alpha}^{2} - | \langle g_{n},Sg_{n}\rangle_{\alpha}|^{2})^{2}} (\|Sg_{n}\|_{\alpha}^{4} + 2 |\langle Sg_{n},g_{n}\rangle_{\alpha}|^{2} + 1).\]
 In order to define $g_{n}$, fix a sequence $\{a_{n}\}_{n=0}^{\infty} \subset \C \setminus \overline{\D}$ such that $\lim_{n\rightarrow \infty} a_{n} = 1$. 
 We first prove the case $\alpha \leq 0$: for $n\in\N$, the function $g_{n}\in D_{\alpha}$ is given by the normalized reproducing kernel in $D_{\alpha}$
 		\[g_{n} = \frac{k_{1/a_{n}}}{\|k_{1/a_{n}}\|_{\alpha}}.\]
 		
 		Using the reproducing property we get that \[\|k_{1/a_{n}}\|_{\alpha}^{2} = \langle k_{1/a_{n}},k_{1/a_{n}}  \rangle_{\alpha} = k_{1/a_{n}}(1/a_{n}) = \sum_{k=0}^{\infty} \frac{|a_{n}|^{-2k}}{(k+1)^{\alpha}}.\]
We can also easily compute, from the definition of the kernel, the square of the norm $	\|Sk_{1/a_{n}}\|_{\alpha}^{2}$, which is equal to
\[\langle Sk_{1/a_{n}},Sk_{1/a_{n}} \rangle_{\alpha}
 		    = |a_{n}|^{2} \sum_{k=1}^{\infty} \frac{|a_{n}|^{-2k}}{k^{2\alpha}} (k+1)^{\alpha} \leq |a_{n}|^{2} \|k_{1/a_{n}}\|_{\alpha}^{2}.\]
 		Hence, we can estimate one of the elements of the matrix $M_n^{-1}$, \[\|Sg_{n}\|_{\alpha}^{2} = \langle Sg_{n},Sg_{n} \rangle_{\alpha}
 		    = \frac{\|Sk_{1/a_{n}}\|_{\alpha}^{2}}{\|k_{1/a_{n}}\|_{\alpha}^{2}} \leq |a_{n}|^{2}.\]
 		At this point, we only lack an estimate of $|\langle Sg_{n}, g_{n} \rangle|^{2}$, but we can actually obtain its exact value by using one more time the reproducing property of the kernel. Indeed, $|\langle Sg_{n}, g_{n} \rangle_{\alpha}|^{2}$ will be given by
\[ \frac{|\langle Sk_{1/a_{n}},k_{1/a_{n}} \rangle_{\alpha}|^{2}}{\|k_{1/a_{n}}\|_{\alpha}^{4}}= \frac{1}{|a_{n}|^{2}} \frac{ |k_{1/a_{n}}(1/a_{n})|^{2}}{\|k_{1/a_{n}}\|_{\alpha}^{4}}= \frac{1}{|a_{n}|^{2}}.\]
Therefore we can obtain a lower bound for the norm of $M_{n}^{-1}$,
 \[\|M_{n}^{-1}\|_{F}^{2} \geq \frac{1}{(|a_{n}|^{2} -  (1/|a_{n}|^{2}))^{2}}  (2 (1/|a_{n}|^{2}) + 1).\]
The right-hand side diverges to $\infty$ if $a_{n}\rightarrow 1$. Thus, taking $a_{n} =1 + \frac{1}{n}$ the result follows.
 		
 		We turn now to the case $\alpha >0$. Consider $f_{n}(z)= \frac{1}{1-z/a_{n}}$ and let  \[g_{n}=\frac{f_{n}}{\|f_{n}\|_{\alpha}}.\] Observe that \[ \|f_{n}\|_{\alpha}^{2}  = \sum_{k=0}^{\infty} |a_{n}|^{-2k} (k+1)^{\alpha}
            = |a_{n}|^{2} \sum_{k=1}^{\infty} |a_{n}|^{-2k} k^{\alpha} 
            = |a_{n}|^{2} \, {\rm Li}_{-\alpha} (|a_{n}|^{-2}),\]
 		where ${\rm Li}_{s} (z)$ is the polylogarithm of order $s$ and argument $z$.
 
An easy calculation shows that $\|Sf_{n}\|_{\alpha}^{2} = |a_{n}|^{2} (\|f_{n}\|_{\alpha}^{2}-1)$ and $\langle Sf_{n},f_{n} \rangle_{\alpha} = a_{n}(\|f_{n}\|_{\alpha}^{2}-1)$. Replacing these expressions in the definition of $g_{n}$ we obtain directly, 
 \[\|Sg_{n}\|_{\alpha}^{2} = |a_{n}|^{2} \left( 1- \frac{1}{\|f_{n}\|_{\alpha}^{2}} \right),\] and the very similar \[  |\langle Sg_{n},g_{n} \rangle_{\alpha}|^{2}    = |a_{n}|^{2} \left( 1- \frac{1}{\|f_{n}\|_{\alpha}^{2}}\right)^{2}.\]
The numerator in the expression for $\|M_n^{-1}\|_F^2$ is greater than $1$ for all $n\in \N$ and the denominator can be expressed in terms of $\|f_{n}\|^{2}_{\alpha}$ as
\[ |a_{n}|^{2} \left( 1- \frac{1}{ \|f_{n}\|_{\alpha}^{2}} \right) \frac{1}{\|f_{n}\|_{\alpha}^{2}} = \frac{1}{{\rm Li}_{-\alpha}(|a_{n}|^{-2})} \left( 1- \frac{1}{|a_{n}|^{2}{\rm Li}_{-\alpha}(|a_{n}|^{-2})}\right).\]	Since $k^{\alpha} \geq 1$ for all $k\geq 1$ and $\alpha>0$ we have \[ {\rm Li}_{-\alpha} (|a_{n}|^{-2}) = \sum_{k=1}^{\infty} |a_{n}|^{-2k} k^{\alpha}  \geq \sum_{k=1}^{\infty} |a_{n}|^{-2k} = \frac{|a_{n}|^{2}}{1-|a_{n}|^{-2}} = \frac{|a_{n}|^{4}}{|a_{n}|^{2}-1},\]
 which tends to $\infty$ if  $a_{n}\rightarrow 1$.
 This implies that $\|M_{n}^{-1}\|_{F}$ diverges as $n$ tends to $\infty$.
 		\end{proof}

{\bf Acknowledgements.} We thank an anonymous referee for careful reading of an earlier version of the manuscript and useful suggestions.

\end{document}